\documentclass{amsart}

\usepackage{mathrsfs}
\usepackage{bbm}
\usepackage{mathtools}
\usepackage{xcolor}
\usepackage{enumitem}

\newtheorem{Thm}{Theorem}[section]
\newtheorem{Lem}[Thm]{Lemma}

\newtheorem{Conj}[Thm]{Conjecture}

\theoremstyle{definition}

\newtheorem{Rem}[Thm]{Remark}

\begin{document}

\title[the sum of a prime and a square-full number in short intervals]{Asymptotic formula for \\ the sum of a prime and a square-full number \\ in short intervals shorter than $X^{1/2}$}
\author[F.Ogihara]{Fumi Ogihara}
\maketitle\thispagestyle{empty}

\begin{abstract}
	Let $R(N)$ be the number of representations of $N$ as a sum of a prime and a square-full number weighted with logarithmic function. 
	In $2024$, the author and Y. Suzuki obtained an asymptotic formula for the sum of $R(N)$ over positive integers $N$ in a short interval ($X$, $X+H$] for $X^{\frac{1}{2}+\varepsilon} \le H < X^{1-\varepsilon}$. 
	In this article, we improve the range of $H$, that is,  we prove the same asymptotic formula for $X^{\frac{32-4\sqrt{15}}{49}+\varepsilon} \le H \le X^{1- \varepsilon}$. 
	
\end{abstract}

\section{Introduction}
	The following is one of the well-known Hardy--Littlewood conjectures. 
	\begin{Conj} [Hardy--Littlewood~\cite{H-L}] \label{Conj_H-L}
		Every sufficiently large positive integer $N$ is a square or the sum of a prime and a square. In other words, 
	$$
		N = n^{2} \quad \mbox{or} \quad p + n^{2}, 
	$$
	where $p$ denotes a prime number and $n$ denotes a positive integer.
	\end{Conj}
	To approach to this conjecture, we consider the following function: 
	\[
		R_{\textup{HL}}(N) \coloneqq \sum_{p+n^{2}=N} \log p. 
	\]
	We call such a function a representation function.
	
	In \cite{H-L}, G. H. Hardy--J. E. Littlewood also conjectured the asymptotic formula: 
	\[
		R_{\textup{HL}}(N)
		\sim
		\mathfrak{S}(N)\sqrt{N}, 
		\quad 
		\textup{($N$: not square)}, 
	\]
	where
	\[
		\mathfrak{S} (N) \coloneqq \prod_{p > 2} \bigg( 1 - \frac{(N/p)}{p-1} \bigg), 
		\quad 
		(N/p)\textup{: Legendre symbol}. 
	\]
	
	In 2017, A. Languasco and A. Zaccanini considered the short interval average of the representation function: 
	\begin{equation*} 
		\sum_{X < N \le X+H} R_{\textup{HL}}(N)
	\end{equation*}
	for real numbers $X$, $H$ with $ 4 \le H \le X$. 
	Let 
	\[
		C \coloneqq \exp \bigg( -c \bigg(\frac{\log X}{\log \log X} \bigg)^{\frac{1}{3}} \bigg), 
	\]
	where $c$ is some small positive constant which may depend on $\varepsilon$.

	\begin{Thm}[{Languasco--Zaccagnini~\cite{L-Z}}] \label{Thm_L-Z}
		For real numbers $X$, $H$ with $4 \le H \le X$ and any real number $\varepsilon >0$, we have
		\begin{equation} \label{L-Zresult}
			\sum_{X < N \le X+H} R_{\textup{HL}}(N)= HX^{\frac{1}{2}}+O( HX^{\frac{1}{2}}C^{-1})
		\end{equation}
		provided 
		\begin{equation*} 
			X^{\frac{1}{2}}C^{-1} \le H \le X^{1- \varepsilon},
		\end{equation*}
		where the implicit constant depend only on $\varepsilon$.
	\end{Thm}
	
	If we obtain (\ref{L-Zresult}) for $H=1$ in Theorem \ref{Thm_L-Z}, Conjecture \ref{Conj_H-L} is true. 
	Therefore, our goal is an improvement of the condition of $H$. 
	
	In 2023, Y. Suzuki improved the range of $H$ in Theorem \ref{Thm_L-Z}. 
	\begin{Thm} [{Suzuki~\cite{Suzuki}}]
	The asymptotic formula \eqref{L-Zresult}
	provided
	\[
	X^{\frac{32-4 \sqrt{15}}{49} + \varepsilon } \le H \le X^{1- \varepsilon}.
	\]
	\end{Thm}
	
	In this article, we study the following function that replaces squares in $R_{\textup{HL}}(N) = \sum_{p+n^{2}=N} \log p$ with square-full numbers: 
	$$
		R(N) \coloneqq \sum_{\substack{p+f=N \\ f \in \mathscr{Q} }} \log p,
	$$
	in which $\mathscr{Q} \coloneqq \{ n \in \mathbb{N} \, | \, n \mbox{ is square-full numbers} \} $.
	Recall that a square-full number is a positive integer $n$ such that if $p$ is a prime dividing $n$, then $p^{2}$ divides $n$. 
	
	The author and Y. Suzuki obtained the following asymptotic formula.

	\begin{Thm} [\cite{O-S}]
		For $A\ge1$, $\varepsilon>0$, and $X^{\frac{1}{2}+\varepsilon} \le H \le X^{1- \varepsilon}$, we have
		\begin{equation} \label{result}
			\sum_{X < N \le X+H} R(N) = \frac{\zeta(\frac{3}{2})}{\zeta(3)}HX^{\frac{1}{2}} \Big( 1 + O \big( (\log X)^{-A} \big) \Big) ,
		\end{equation}
		where the implicit constant depends on $A,\varepsilon$.
	\end{Thm}
	Our main theorem is the following improvement of the  above range of $H$. 
	\begin{Thm} \label{main}
		For $\varepsilon>0$, we have the asymptotic formula (\ref{result}) provided 
		$$
			X^{\frac{32-4\sqrt{15}}{49}+\varepsilon} \le H \le X^{1- \varepsilon}. 
		$$
	\end{Thm}

\section{Preliminary lemmas}
	In this section, we prepare some lemmas.
	
	For a real number $B \ge 2$, we define
	\begin{equation*}
			\widetilde{R}_{B}(N) : = \sum_{\substack{p + a^{2}b^{3} = N \\ b \le B}} (\log p) \mu(b)^{2},
	\end{equation*}
    where $\mu(n)$ is a M\"{o}bius function of $n$. 
	
	\begin{Lem} [\cite{O-S}] \label{Lem_R(N)=R_{B}(N)+(error)}
		For a positive number $N$ and a real number $B \ge 1$, we have
		$$
			R(N) = \widetilde{R}_{B}(N) + O( (N^{\frac{1}{2}} \log N ) B^{- \frac{1}{2}} ).
		$$
	\end{Lem}
	
	By Lemma \ref{Lem_R(N)=R_{B}(N)+(error)}, we have
	$$
		\sum_{X < N \le X+H} R(N) = \sum_{X \le N \le X+H} \widetilde{R}_{B}(N) + O \Bigg( \sum_{X < N \le X+H} (N^{\frac{1}{2}} \log N ) B^{- \frac{1}{2}} \Bigg).
	$$
	If $B = (\log X)^{4A}$ for $A >1$, we have
	\begin{equation}
	\label{R_to_RB}
		\sum_{X < N \le X+H} R(N) = \sum_{X \le N \le X+H} \widetilde{R}_{B}(N) + O( HX^{\frac{1}{2}}(\log X)^{-A}).
	\end{equation}
	We consider
	the sum of $\widetilde{R}_{B}(N)$
	on the right-hand side.
	
	Let $\mathscr{Q}_{B} \coloneqq \{ n \in \mathbb{N} \, | \, n=a^{2}b^{3}, b \le B \mbox{ and } b \mbox{ is square-free} \}$. Then, we have
	\begin{align*}
		\sum_{X < N \le X+H} \widetilde{R}_{B}(N) 
		=& \sum_{\substack{X < p+f \le X+H \\ f \in \mathscr{Q}_{B}}} \log p . 
	\end{align*}
	
	To evaluate the short interval sum with square-full numbers, we introduce some lemmas. 
		\begin{Lem} [\cite{O-S}] \label{Lem_B}
        Let $B$ be a real number greater than or equal to $1$. 
		For real numbers $X$, $H$ with $4 \le H \le X$ and $n^{2} \le X+H$, we have 
		$$
		\sum_{\substack{n^{2} \le f < (n+1)^{2} \\ f \le X+H \\ f \in \mathscr{Q}_{B}}} 1 
		\ll B.
		$$
	\end{Lem}
	
	\begin{Lem} \label{Lem_sum_mobius}
		For real numbers $x \le 1$ and $s >1$, we have
		\[
			\sum_{n \le x} \frac{\mu(n)^{2}}{n^{2}}
			= \frac{\zeta(s)}{\zeta(2s)} + O(x^{-s+1}),
		\]
        where $\zeta(s)$ is the Riemann zeta function. 
	\end{Lem}
	\begin{proof}
		We can write 
		\begin{equation} \label{sum_mobius}
			\sum_{n \le x} \frac{\mu(n)^{2}}{n^{2}}
			= \sum_{n=1}^{\infty} \frac{\mu(n)^{2}}{n^{2}}
			+ \sum_{n > x} \frac{\mu(n)^{2}}{n^{2}}.
		\end{equation}
		From Euler Product Formula for the zeta function, we have
		\begin{align*}
			\sum_{n=1}^{\infty} \frac{\mu(n)^{2}}{n^{2}}
			& = \zeta(s) \prod_{p : prime} \bigg( 1+ \frac{\mu(p)^{2}}{p^{s}} \bigg) \bigg( 1- \frac{1}{p^{s}} \bigg) \\
			& = \zeta(s) \prod_{p : prime} \bigg( 1 - \frac{1}{p^{2s}} \bigg) 
			= \frac{\zeta(s)}{\zeta(2s)}.
		\end{align*}
		The second term on the right-hand side
	of (\ref{sum_mobius}) is
		\[
			\sum_{n > x} \frac{\mu(n)^{2}}{n^{2}}
			\le \int_{x}^{\infty} \frac{1}{u^{s}} du 
			\ll x^{-s+1}.
		\]
		Substituting these formulas into (\ref{sum_mobius}) leads to the lemma.
	\end{proof}
	
	We first estimate the sum of the square-full numbers in the short intervals. 
	
	\begin{Lem} \label{Lem_5}
        Let $B$ be a real number greater than or equal to $1$. 
		For real numbers $X$, $H \ge 2$, we have
		\[
			\sum_{\substack{X < f \le X+H \\ f \in \mathscr{Q}_{B}}} 1 
			\ll HX^{-\frac{1}{2}}B + B.
		\]
	\end{Lem}
	
	\begin{proof}
		We write the left-hand side as
		\[
			\sum_{\substack{X < f \le X+H \\ f \in \mathscr{Q}_{B}}} 1 
			= \sum_{X < n^{2} \le X+H} \sum_{\substack{n^{2} \le f < (n+1)^{2} \\ f \le X+H \\ f \in \mathscr{Q}_{B}}} 1.
		\]
		By Lemma \ref{Lem_B}, we have
		\begin{align*}
			\sum_{\substack{X < f \le X+H \\ f \in \mathscr{Q}_{B}}} 1 
			&\ll B \sum_{X < n^{2} \le X+H} 1 \\
			& \le B( (X+H)^{\frac{1}{2}} - X^{\frac{1}{2}}) + 1 \\
			& = B \bigg( \frac{1}{2} \int_{X}^{X+H} u^{-\frac{1}{2}} du + 1 \bigg)
			\le HX^{-\frac{1}{2}}B +B.
		\end{align*}
		This completes the proof.
	\end{proof}
		
	Let 
	\[
		S(X) \coloneqq \sum_{\substack{f \le X \\ f \in \mathscr{Q}_{B}}} (X-f).
	\]
	
	\begin{Lem} \label{Lem_7}
		For real numbers $X, H$ with $4 \le H \le X$, we have
		\[
			S(X+H) - S(X)
			= \frac{\zeta(\frac{3}{2})}{\zeta(3)} HX^{\frac{1}{2}} + O(HB).
		\]
	\end{Lem}
	\begin{proof}
		From the definition of $\mathscr{Q}_{B}$, we have
		\begin{align*}
			S(X+H) - S(X)
			&= H \sum_{\substack{a^{2}b^{3} \le X \\ b < B}} \mu(b)^{2} \\
			&= H \sum_{b<B} \mu(b)^{2} \sum_{a^{2} \le \frac{X}{b^{3}}} 1 \\
			&= H \bigg( X^{\frac{1}{2}} \sum_{b < B} \frac{\mu(b)^{2}}{b^{\frac{3}{2}}} + O (B) \bigg). 
		\end{align*}
		Using Lemma \ref{Lem_sum_mobius}, we obtain the lemma.
	\end{proof}
	
	We second consider the Chebyshev functions $\psi(x)$ and $\vartheta(x)$. 
	Let $\Lambda(m)$ be the von Mangoldt function. 
	Chebyshev functions are defined as
	\[
		\psi(x) \coloneqq \sum_{m \le x} \Lambda(m), 
		\quad
		\vartheta(x) \coloneqq \sum_{\substack{p : \textup{prime} \\ p \le x}} \log p.
	\]
	We denote non-trivial zeros of the Riemann zeta function as $\rho = \beta + i\gamma$ with the real part $\beta$ and the imaginary part $\gamma$.
	Let $N(\alpha, T)$ be the number of non-trivial zeros of the Riemann zeta function in the range of $\alpha \le \beta \le 1$ and $|\gamma| \le T$. 
		
	\begin{Lem} [Suzuki~\cite{Suzuki}] \label{Lem_9}
		Assume that real numbers $X, T$ satisfy $2 \le T \le 2X$. Then for $0 \le x \le X$, we have
		\begin{equation} \label{assym_psi}
			\psi(x) 
			= x 
			- \sum_{\substack{\rho \\ | \gamma | \le T}} \frac{x^{\rho}}{\rho} 
			+ O( XT^{-1}(\log X)^{2}).
		\end{equation}
	\end{Lem}
	
	\begin{Lem} [{\cite[Theorem 6.1, p. 143]{Ivic}}]{\textup{(The Korobov--Vinogradov zero-free region)}} \label{Lem_zerofree}
		For $s = \sigma +it$, $\tau = |t| +4$, $ \sigma > 1 - c_{0}(\log \tau)^{-\frac{2}{3}} (\log \log \tau)^{-\frac{1}{3}}$, we have $\zeta(s) \ne 0$, where $c_{0} > 0$ is some absolute constant. 
	\end{Lem}
	
	\begin{Lem} [Ingham~\cite{Ingham} and Huxley~\cite{H}, {\cite[p. 275]{Ivic}}]\label{Lem_I-H}
		For real numbers $\alpha, T$ with $\frac{1}{2} \le \alpha \le 1$ and $T \ge 2$, we have
		\[
			N(\alpha, T) \ll T^{c(\alpha)}(\log T)^{A}, 
			\quad
			c(\alpha) \coloneqq
			\left\{ \,
				\begin{aligned}
					& \frac{3(1-\alpha)}{3\alpha-1} & (\frac{3}{4} \le \alpha \le 1), \\
					& \frac{3(1-\alpha)}{2-\alpha} & (\frac{1}{2} \le \alpha \le \frac{3}{4}), 
				\end{aligned}
			\right.
		\]
		where $A$ is a constant.
	\end{Lem}

 	 \begin{Rem}
 		 Recently this result has improved by L. Guth and J. Maynard \cite{G-M}. 
		 Using their results does not affect our main theorem for now. 
	\end{Rem}

	In the following, we evaluate several sums over non-trivial zeros of the Riemann zeta function. 
	
	\begin{Lem} [Suzuki~\cite{Suzuki}] \label{Lem_13}
		For real numbers $K, Y$ with $1 \le K \le Y \le X^{2}$, we have
		\[
			\sum_{\substack{\rho \\ K \le | \gamma | \le 2K}} Y^{\beta} 
			\ll ( Y^{\phi(\lambda)} + Y^{1-\eta + 2\eta \lambda}) (\log X)^{A},
		\]
		where
		\begin{align}
			\phi(\lambda) \coloneqq 
			\left\{ \,
				\begin{aligned}
					&  \frac{3}{5} \lambda + \frac{3}{4} & (0 \le \lambda \le \frac{25}{48}), \\
					& 3 \lambda + 2(1- \sqrt{3\lambda}) & (\frac{25}{48} \le \lambda \le \frac{3}{4}), \\
					& \lambda + \frac{1}{2} & (\frac{3}{4} \le \lambda \le 1),  
				\end{aligned}
			\right.
			\label{fun_phi} \\
			\eta = c_{1}(\log X)^{-\frac{2}{3}} (\log \log X)^{-\frac{1}{3}}, \quad
		\lambda = \frac{\log K}{\log Y}, \notag
		\end{align}
		and constants $A, c_{1} >0$.
	\end{Lem}
	
	\begin{Lem} [Suzuki~\cite{Suzuki}] \label{Lem_14}
		For a real number $\lambda$ with $0 \le \lambda \le 1$, let $\phi(\lambda)$ be the function given by (\ref{fun_phi}), Then we have
		\begin{equation} \label{differ_psi}
			\frac{3}{5} \le \phi ' (\lambda) \le 1. 
		\end{equation}
	\end{Lem}
	
	\begin{Lem} [Suzuki~\cite{Suzuki}] \label{Lem_16}
		Let $\phi(\lambda)$ be the function given by (\ref{fun_phi}) and consider following equations:
		\[
			\phi(\lambda_{1}) - \frac{1}{2} \lambda_{1} =1, 
			\quad
			\phi(\lambda_{2}) + \frac{1}{2} \lambda_{2} = \frac{3}{2}.
		\]
		The solutions $\lambda_{1}$ and $\lambda_{2}$ of these equations are
		\begin{equation} \label{lambda1_lambda2}
			\lambda_{1} = 1, 
			\quad
			\lambda_{2} = \frac{17+4\sqrt{15}}{49}.
		\end{equation}
	\end{Lem}
	
	\begin{Lem} [Suzuki~\cite{Suzuki}]
		For a real number $\varepsilon$ and $0 \le \lambda \le \frac{17+4\sqrt{15}}{49} - \varepsilon$, we have
		\[
			\phi(\lambda) - \frac{1}{2} \lambda \le 1 -\frac{\varepsilon}{10}, 
			\quad
			\phi(\lambda) + \frac{1}{2} \lambda \le \frac{3}{2} -\frac{\varepsilon}{10}.
		\]
	\end{Lem}
	
	To estimate the resulting exponential integrals, we recall two standard estimates. 
	For a complex valued function $f$ defined over an interval $[ a, b ]$, let $V[a,b](f)$ be the total variation of $f$ over $[a, b]$ and
		\[
			||f|| 
			= ||f||_{BV([a, b])}
			= \sup_{x \in [a, b]} |f(x)| + V_{[a, b]} (f). 
		\]
	For a real number $x$, let $e(x)=\exp(2 \pi ix)$. 
	
	\begin{Lem} [{\cite[Lemma 2.1, p. 56]{Ivic}}]\label{Lem_18}
		For a real number $\lambda >0$ and $f$, $g$ be real-valued functions defined over an interval $[a, b]$ satisfying
		\begin{enumerate}[label=\rm{(\Alph*)}]
			\item f is continuously differentiable on the interval $[a, b]$, 
			\item $f'$ is monotonic on the interval $[a,b]$, 
			\item $f'$ satisfies $|f'(x)| \le \lambda$ on the interval $[a, b]$. 
		\end{enumerate}
		Then, we have 
		\[
			\int_{a}^{b} g(x) e(f(x)) dx 
			\ll ||g|| \lambda ^{-1}.
		\]
	\end{Lem}
	
	\begin{Lem} [{\cite[Lemma 2.2, p. 56]{Ivic}}]\label{Lem_19}
		For a real number $\lambda >0$ and $f$, $g$ be real-valued functions defined over an interval $[a, b]$ satisfying
		\begin{enumerate}[label=\rm{(\Alph*)}]
			\item $f$ is twice continuously differentiable on the interval $[a, b]$, 
			\item $f''$ satisfies $|f''(x)| \le \lambda$ on the interval $[a, b]$. 
		\end{enumerate}
		Then, we have 
		\[
			\int_{a}^{b} g(x) e(f(x)) dx 
			\ll ||g|| \lambda ^{-\frac{1}{2}}.
		\]
	\end{Lem}

\section{proof of the main theorem}
	
	This section is devoted to the proof of Theorem \ref{main}. 
	Before that, we prepare some lemmas.
		
	\begin{Lem} \label{Lem_20}
		For real numbers $X$, $H$, $\varepsilon$ with $4 \le H \le X$ and $\varepsilon >0$, we have
		\[
			\sum_{X \le N \le X+H} \widetilde{R}_{B}(N) 
			= \sum_{\substack{X< m+f \le X+H \\ f \in \mathscr{Q}_{B}}} \Lambda (m) +O(HX^{\frac{1}{2}}C^{-1}). 
		\]
	\end{Lem}
	\begin{proof}
		By the definition of $\widetilde{R}_{B}(N)$, 
		\[
			\sum_{X \le N \le X+H}\widetilde{R}_{B}(N) 
			= \sum_{\substack{X < f \le X+H \\ f \in \mathscr{Q}_{B}}} \sum_{X-f < p \le X+H-f} \log p.
		\]
		From $\psi(x) = \vartheta(x) + O(x^{\frac{1}{2}})$ and Lemma \ref{Lem_5}, 
		\begin{align*}
			\sum_{\substack{X < f \le X+H \\ f \in \mathscr{Q}_{B}}} \widetilde{R}_{B}(N)
			&= \sum_{\substack{X < p+f \le X+H \\ f \in \mathscr{Q}_{B}}} \Lambda(m) + O \bigg(X^{\frac{1}{2}} \sum_{\substack{X < f \le X+H \\ f \in \mathscr{Q}_{B}}} 1 \bigg) \\
			&= \sum_{\substack{X < p+f \le X+H \\ f \in \mathscr{Q}_{B}}} \Lambda(m) + O(HB) \\
			&= \sum_{\substack{X < p+f \le X+H \\ f \in \mathscr{Q}_{B}}} \Lambda(m) + O(HX^{\frac{1}{2}}C^{-1}).
		\end{align*}
		We odtain the lemma.
	\end{proof}

	\begin{Lem} \label{Lem_21}
		For real numbers $X$, $H$, $\varepsilon$ with $4 \le H \le X$ and $\varepsilon >0$, we have
		\[
			\sum_{X \le N \le X+H} \widetilde{R}_{B}(N) 
			= \sum_{\substack{f \le X \\ f \in \mathscr{Q}_{B}}} \bigg( \psi(X+H-f) - \psi(X-f) \bigg) +O(HX^{\frac{1}{2}}C^{-1}). 
		\]
	\end{Lem}
	\begin{proof}
		From Lemma \ref{Lem_20},
		\begin{equation} \label{use_Lem20}
			\sum_{X \le N \le X+H} \widetilde{R}_{B}(N) 
			= \sum_{\substack{X< m+f \le X+H \\ f \in \mathscr{Q}_{B}}} \Lambda (m) +O(HX^{\frac{1}{2}}C^{-1}) . 
		\end{equation}
		The first term on the right hand side of (\ref{use_Lem20}) is 
		\begin{align*}
			\sum_{\substack{X< m+f \le X+H \\ f \in \mathscr{Q}_{B}}} \Lambda (m)
			&= \sum_{\substack{f \le X+H \\ f \in \mathscr{Q}_{B}}} \sum_{X-f < m \le X+H-f} \Lambda (m) \\
			&= \sum_{\substack{f \le X \\ f \in \mathscr{Q}_{B}}} \bigg( \psi(X+H-f) - \psi(X-f) \bigg) \\
			&\hspace{5em} + \sum_{\substack{X < f \le X+H \\ f \in \mathscr{Q}_{B}}} \bigg( \psi(X+H-f) - \psi(X-f) \bigg) . 
		\end{align*}
		By $\psi(x) \ll x$ and Lemma \ref{Lem_5}, 
		\begin{align*}
			\sum_{\substack{X < f \le X+H \\ f \in \mathscr{Q}_{B}}} \bigg( \psi(X+H-f) - \psi(X-f) \bigg) 
			&\ll H \sum_{\substack{X < f \le X+H \\ f \in \mathscr{Q}_{B}}} 1 \\
			&= H (HX^{-\frac{1}{2}}B + B) \\
			&\ll HX^{\frac{1}{2}}C^{-1} . 
		\end{align*}
		This gives the assertion.
	\end{proof}

	Let
			\[
			S_{\alpha} (Q) \coloneqq \frac{1}{\alpha} \sum_{\substack{f \le X \\ f \in \mathscr{Q}_{B}}} (Q-f)^{\alpha},
			\quad
			S(Q) = S_{1} (Q).
		\]
	
	\begin{Lem} \label{Lem_22}
		For real numbers $X$, $H$, $Q$, $T$, with $4 \le H \le X$, $X \le Q \le X+H$ and $2 \le T \le X$, we have
		\[
			\sum_{\substack{f \le X \\ f \in \mathscr{Q}_{B}}} \psi(Q-f)
			= S(Q) - \sum_{\substack{\rho \\ | \gamma | \le T}} S_{\rho}(Q) + O(X^{\frac{3}{2}}T^{-1} (\log X)^{2} ).
		\]
	\end{Lem}
	\begin{proof}
		From Lemma \ref{Lem_9}, this immediately follows. 
	\end{proof}
	
	\begin{Lem} [Suzuki~\cite{Suzuki}] \label{Lem_23}
		For a integer $n$ and real numbers $\alpha, \gamma, Q, U, V$ with $\alpha \le 1$, $|\gamma| \ge 1$ and $1 \le U \le V \le Q$, we have
		\[
			\int_{U}^{V} u^{\alpha +i\gamma -1} e \bigg( n(Q-u)^{\frac{1}{2}} \bigg) du 
			\ll
			\left\{ \,
				\begin{aligned}
					& \frac{V^{\alpha} \log X}{|\gamma|^{\frac{1}{2}}} & (\alpha \ge 0), \\
					& \frac{U^{\alpha} \log X}{|\gamma|^{\frac{1}{2}}} & (\alpha \le 0), \\
					& \frac{Q^{\frac{1}{2}}}{|n|} & (|n| > 2Q^{\frac{1}{2}} |\gamma|).
				\end{aligned}
			\right.
		\]
	\end{Lem}
	
	\begin{Lem} \label{Lem_24}
		For real numbers $X$, $H$, $\varepsilon$ with $4 \le H \le X$ and $\varepsilon >0$ and a non-trivial zero $\rho = \beta +i\gamma$ of $\zeta(s)$ with $|\gamma| \le 2X$, we have
		\begin{align*}
			S_{\rho}(X+H) -S_{\rho}(X) 
			\ll & \frac{\Gamma(\rho) \Gamma(\frac{1}{2})}{\Gamma(\rho +\frac{3}{2})} \bigg( (X+H)^{\rho +\frac{1}{2}} - X^{\rho + \frac{1}{2}} \bigg)
			+ \frac{(X+H)^{\rho} -X^{\rho}}{2\rho} B \\
			& + H^{\beta} |\gamma|^{\beta -\frac{1}{2}} (\log X)^{2} 
			+ \frac{HX^{\frac{1}{2}}C^{-2}}{|\gamma|}
			+ B \log X 
		\end{align*}
		provided
		\[
			X^{\varepsilon} \le H \le X^{1- \varepsilon}, 
		\]
        where $\Gamma(n)$ is the gamma function of $n$. 
	\end{Lem}
	\begin{proof}
		For $X \le Q \le X+H$,
		\begin{align}
			S_{\rho}(Q)
			&= \frac{1}{\rho} \sum_{b \le B } \mu(b)^{2} \sum_{a^{2} \le \frac{X}{b^{3}}} (Q-a^{2}b^{3})^{\rho} \notag \\
			&= \frac{1}{\rho} \sum_{b \le B } \mu(b)^{2}  \int_{0}^{\frac{X}{b^{3}}} (Q-b^{3}u)^{\rho} d[u^{\frac{1}{2}}] \notag \\
			&=  \sum_{b \le B } \mu(b)^{2} \bigg( \frac{1}{2\rho} \int_{0}^{\frac{X}{b^{3}}} (Q-b^{3}u)^{\rho} u^{-\frac{1}{2}} du
			+ \frac{1}{\rho}  \int_{0}^{\frac{X}{b^{3}}} (Q-b^{3}u)^{\rho} d \bigg( \{ u^{\frac{1}{2}} \}- \frac{1}{2} \bigg) \bigg). \label{S_rho(Q)} 
		\end{align}
		
		First, we consider the first term on the most right-hand side of (\ref{S_rho(Q)}). 
		Changing the variables and property of the beta function, $B(x, y) = \int_{0}^{1} t^{x-1} (1-t)^{y-1} \, dt$, gives 
		\begin{align*}
			& \frac{1}{2\rho} \int_{0}^{\frac{X}{b^{3}}} (Q-b^{3}u)^{\rho} u^{-\frac{1}{2}} du \\
			=& \frac{1}{2\rho} b^{-\frac{3}{2}} \int_{0}^{Q} (Q-u)^{\rho} u^{-\frac{1}{2}} du 
			+ \frac{1}{2\rho} b^{-\frac{3}{2}} \int_{Q}^{X} (Q-u)^{\rho} u^{-\frac{1}{2}} du \\
			=& \frac{1}{2} \frac{\Gamma(\rho) \Gamma(\frac{1}{2})}{\Gamma(\rho +\frac{3}{2})} b^{-\frac{3}{2}} Q^{\rho +\frac{1}{2}} 
			+ O \bigg( b^{-\frac{3}{2}} \frac{H^{2}X^{-\frac{1}{2}}}{|\gamma|} \bigg) \\
			=& \frac{1}{2} \frac{\Gamma(\rho) \Gamma(\frac{1}{2})}{\Gamma(\rho +\frac{3}{2})} b^{-\frac{3}{2}} Q^{\rho +\frac{1}{2}} 
			+ O \bigg( b^{-\frac{3}{2}} \frac{HX^{\frac{1}{2}}B^{-1}C^{-2}}{|\gamma|} \bigg). 
		\end{align*}
		
		Next, we consider the second term on the most-right-hand side of (\ref{S_rho(Q)}).
		By integration by parts, we have 
		\begin{align*}
			& \frac{1}{\rho} \int_{0}^{\frac{X}{b^{3}}} (Q-b^{3}u)^{\rho} d \bigg( \{ u^{\frac{1}{2}} \}- \frac{1}{2} \bigg) \\
			=& -b^{3} \int_{0}^{\frac{X}{b^{3}}} (Q-b^{3}u)^{\rho -1} \bigg( \{ u^{\frac{1}{2}} \}- \frac{1}{2} \bigg) du 
			- \frac{Q^{\rho}}{2\rho}
			+ O \bigg( \frac{H}{|\gamma|} \bigg) \\
			=& -b^{3} \int_{0}^{\frac{X}{b^{3}} -\frac{1}{b^{3}}} (Q-b^{3}u)^{\rho -1} \bigg( \{ u^{\frac{1}{2}} \}- \frac{1}{2} \bigg) du \\
			&-b^{3} \int_{\frac{X}{b^{3}} -\frac{1}{b^{3}}}^{\frac{X}{b^{3}}} (Q-b^{3}u)^{\rho -1} \bigg( \{ u^{\frac{1}{2}} \}- \frac{1}{2} \bigg) du 
			- \frac{Q^{\rho}}{2\rho}
			+ O \bigg( \frac{HX^{\frac{1}{2}}B^{-1}C^{-2}}{|\gamma|} \bigg). 
		\end{align*}
		The second integral of most-right-hand side of above is
		\[
			-b^{3} \int_{\frac{X}{b^{3}} -\frac{1}{b^{3}}}^{\frac{X}{b^{3}}} (Q-b^{3}u)^{\rho -1} \bigg( \{ u^{\frac{1}{2}} \}- \frac{1}{2} \bigg) du
			\ll \frac{1}{\beta}.
		\]
		From Lemma \ref{Lem_zerofree}, we have $\frac{1}{\beta} \ll \log X$.
		Then,
		\begin{align*}
		& \frac{1}{\rho} \int_{0}^{\frac{X}{b^{3}}} (Q-b^{3}u)^{\rho} d \bigg( \{ u^{\frac{1}{2}} \}- \frac{1}{2} \bigg) \\
		=& -b^{3} \int_{0}^{\frac{X}{b^{3}} -\frac{1}{b^{3}}} (Q-b^{3}u)^{\rho -1} \bigg( \{ u^{\frac{1}{2}} \}- \frac{1}{2} \bigg) du 
		- \frac{Q^{\rho}}{2\rho}
		+ O \bigg( \frac{HX^{\frac{1}{2}}B^{-1}C^{-2}}{|\gamma|} + \log X \bigg). 
		\end{align*}
		By the Fourier expansion:
		\[
			\{ u \} - \frac{1}{2}
			= - \sum_{n \ne 0} \frac{e(nu)}{2\pi in}, 
		\]
		we have
		\[
			\frac{1}{\rho} \int_{0}^{\frac{X}{b^{3}}} (Q-b^{3}u)^{\rho} d \bigg( \{ u^{\frac{1}{2}} \}- \frac{1}{2} \bigg) 
			=- \frac{Q^{\rho}}{2\rho}
			+ b^{3} R_{\rho}(Q)
			+ O \bigg( \frac{HX^{\frac{1}{2}}B^{-1}C^{-2}}{|\gamma|} + \log X \bigg), 
		\]
		where
		\[
			I_{\rho}(Q, n)
			\coloneqq I _{\rho , k, l}(Q, n)
			= \int_{0}^{\frac{X}{b^{3}} -\frac{1}{b^{3}}} (Q-b^{3}u)^{\rho-1} e(nu^{\frac{1}{2}}) du , 
		\]
		\[
			R_{\rho}(Q) 
			\coloneqq R_{\rho , k, l}(Q) 
			= \sum_{n \ne 0} \frac{I_{\rho}(Q, n)}{2\pi in}. 
		\]
        
		We estimate $I_{\rho}(X+H, n) - I_{\rho}(X, n)$ to estimate $R_{\rho}(X+H) - R_{\rho}(X)$.
		Changing the variables gives
		\begin{align*}
			&I_{\rho}(X+H, n)
			= b^{-3} \int_{1}^{X} (u+H)^{\rho-1} e(b^{-\frac{3}{2}}n(X-u)^{\frac{1}{2}}) du , \\
			&I_{\rho}(X, n)
			= b^{-3} \int_{1}^{X} u^{\rho-1} e(b^{-\frac{3}{2}}n(X-u)^{\frac{1}{2}}) du . 
		\end{align*}
		Let $U = \min(4H|\gamma|, X)$. 
		We write
		\[
			I_{\rho}(X+H, n) - I_{\rho}(X, n)
			=b^{-3} (I + I_{1} +I_{2}),
		\]
		where
				\begin{align*}
			&I 
			\coloneqq \int_{U}^{X} ((u+H)^{\rho-1} - u^{\rho-1}) e(b^{-\frac{3}{2}}n(X-u)^{\frac{1}{2}}) du , \\
			&I_{1} 
			\coloneqq \int_{1}^{U} (u+H)^{\rho-1} e(b^{-\frac{3}{2}}n(X-u)^{\frac{1}{2}}) du , \\
			&I_{2} 
			\coloneqq \int_{1}^{U} u^{\rho-1} e(b^{-\frac{3}{2}}n(X-u)^{\frac{1}{2}}) du . 
		\end{align*}
		
		For $I$, by the Taylor expansion: 
		\[
			(u+H)^{\rho -1} 
			= u^{\rho -1} 
			+ \sum_{\nu =1}^{\infty} \binom{\rho -1}{\nu} H^{\nu}, 
		\]
		therefore we have 
		\[
			I 
			\ll \sum_{\nu =1}^{\infty} \binom{\rho -1}{\nu} H^{\nu} \int_{U}^{X} u^{\rho - \nu -1} e \bigg( n(X-u)^{\frac{1}{2}} \bigg) du.
		\]
		In the case of $4H|\gamma| \le X$, by using Lemma \ref{Lem_23} and $U=4H|\gamma|$,  
		\[
			I 
			\ll \frac{U^{\beta} \log X}{|\gamma|^{\frac{1}{2}}} \sum_{\nu =1}^{\infty} \bigg| \binom{\rho -1}{\nu} \bigg| \bigg( \frac{H}{U} \bigg)^{\nu} 
			\ll \frac{U^{\beta} \log X}{|\gamma|^{\frac{1}{2}}} \sum_{\nu =1}^{\infty} \prod_{\mu =1}^{\nu} \frac{|\gamma|+ 2\mu}{4\mu |\gamma|}
			\ll \frac{U^{\beta} \log X}{|\gamma|^{\frac{1}{2}}} . 
		\]
		In the case of $4H|\gamma| \ge X$, $I$ is an empty integral, then we have the same estimate. 
		
		For $I_{1}$ by chang- ing the variable and Lemma \ref{Lem_23}, 
		\[
			I_{1}
			\ll \int_{1+H}^{U+H} u^{\beta +i\gamma -1} \exp(n(X+H-u)^{\frac{1}{2}}) du 
			\ll \frac{U^{\beta} \log X}{|\gamma|^{\frac{1}{2}}} . 
		\]
		Similary, 
		\[
			I_{2}
			\ll \frac{U^{\beta} \log X}{|\gamma|^{\frac{1}{2}}} . 
		\]
		
		Then, we have
		\[
			I_{\rho}(X+H, n) - I_{\rho}(X, n) 
			= b^{-3} \frac{U^{\beta} \log X}{|\gamma|^{\frac{1}{2}}} 
			\ll b^{-3} H^{\beta} |\gamma|^{\beta - \frac{1}{2}} \log X.
		\]
		
		If $|n| > 2(X+H)^{\frac{1}{2}} |\gamma|$, from Lemma \ref{Lem_23}, 
		\[
			I_{\rho}(X+H, n) - I_{\rho}(X, n) 
			\ll \frac{X^{-\frac{1}{2}}}{|n|}. 
		\]
		
		Thus we have
		\begin{align*}
			&R_{\rho}(X+H) - R_{\rho}(X) \\
			= & \sum_{n \le 2(X+H)^{-\frac{1}{2}} |\gamma|} \hspace{-1em} \frac{ I_{\rho}(X+H, n) - I_{\rho}(X, n) }{2 \pi i n} 
			+ \sum_{n > 2(X+H)^{-\frac{1}{2}} |\gamma|} \hspace{-1em} \frac{ I_{\rho}(X+H, n) - I_{\rho}(X, n) }{2 \pi i n} \\
			\ll & b^{-3} \bigg( H^{\beta} |\gamma|^{\beta -\frac{1}{2}} \log X \sum_{n \le 2(X+H)^{-\frac{1}{2}} |\gamma|} \frac{1}{n} 
			+ X^{-\frac{1}{2}} \sum_{n > 2(X+H)^{-\frac{1}{2}} |\gamma|} \frac{1}{n^{2}} \bigg) \\
			\ll & b^{-3} (H^{\beta} |\gamma|^{\beta -\frac{1}{2}} (\log X)^{2} +1). 
			\end{align*}
			
			Therefore, we get 
			\begin{align*}
				&S_{\rho}(X+H) -S_{\rho}(X) \\
				\ll & \sum_{b \le B} \mu(b)^{2} \bigg( \frac{\Gamma(\rho) \Gamma(\frac{1}{2})}{\Gamma(\rho +\frac{3}{2})} b^{-\frac{3}{2}} \big( (X+H)^{\rho +\frac{1}{2}} - X^{\rho + \frac{1}{2}} \big)
			+ \frac{(X+H)^{\rho} -X^{\rho}}{2\rho} \\
			& \hspace{7em} + b^{-3} (H^{\beta} |\gamma|^{\beta -\frac{1}{2}} (\log X)^{2}) 
			+ \frac{HX^{\frac{1}{2}}B^{-1}C^{-2}}{|\gamma|}
			+ \log X \bigg)
		\end{align*}
		Lemma \ref{Lem_sum_mobius} leads to the lemma. 
	\end{proof}
	
	We finally prove Theorem \ref{main}. 
	From Lemmas \ref{Lem_21}, \ref{Lem_22} and \ref{Lem_24}, we have 
	\begin{align*}
		\sum_{X < N \le X+H} \widetilde{R}_{B}(N) 
		=& \sum_{\substack{f \le X \\ f \in \mathscr{Q}_{B}}} \bigg( \psi(X+H-f) - \psi(X-f) \bigg) +O(HX^{\frac{1}{2}}C^{-1}) \\
		=& S(X+H) - \sum_{\substack{\rho \\ | \gamma | \le T}} S_{\rho}(X+H) + O(X^{\frac{3}{2}}T^{-1} (\log X)^{2} ) \\
		&-S(X) + \sum_{\substack{\rho \\ | \gamma | \le T}} S_{\rho}(X) + O(X^{\frac{3}{2}}T^{-1} (\log X)^{2} ) +O(HX^{\frac{1}{2}}C^{-1}) \\
		=& M + O( | R_{1} | + R_{2} + (R_{3} + X^{\frac{3}{2}}T^{-1} +T)(\log X)^{2} +HX^{\frac{1}{2}}C^{-1} )
	\end{align*}
	provided
	\begin{equation} \label{condition1}
		X^{\varepsilon} \le H \le X^{1- \varepsilon}, 
		\quad
		2 \le T \le X, 
	\end{equation}
	where
	\[
		M \coloneqq S(X+H) - S(X), 
		\quad
		R_{1} \coloneqq \sum_{\substack{\rho \\ | \gamma | \le T}} \frac{\Gamma(\rho) \Gamma(\frac{1}{2})}{\Gamma(\rho +\frac{3}{2})} \bigg( (X+H)^{\rho +\frac{1}{2}} - X^{\rho + \frac{1}{2}} \bigg), 
	\]
	\[
		R_{2} \coloneqq \sum_{\substack{\rho \\ | \gamma | \le T}} \frac{(X+H)^{\rho} -X^{\rho}}{2\rho} B, 
		\quad
		R_{3} \coloneqq \sum_{\substack{\rho \\ | \gamma | \le T}} H^{\beta} |\gamma|^{\beta -\frac{1}{2}}. 
	\]
	
	In order to control the size of $X^{\frac{3}{2}}T^{-1} (\log X)^{2}$, put
	\begin{equation} \label{T}
		T = X^{1+\varepsilon_{1}}H^{-1}, 
		\quad
		0 < \varepsilon_{1} \le \frac{\varepsilon}{2}. 
	\end{equation}
	From (\ref{condition1}), we have
	\begin{equation} \label{T_condition}
		X^{\varepsilon} \le T \le X^{1- \frac{\varepsilon}{2}}.
	\end{equation}
	
	If $X^{\frac{1}{4} +\varepsilon} \le H$, 
	\[
		T(\log X)^{2}
		= HX^{1+ \varepsilon_{1}}H^{-2}(\log X)^{2}
		\ll HX^{\frac{1}{2}}C^{-1}. 
	\]
	Therefore we get 
	\[
		(X^{\frac{3}{2}}T^{-1} +T)(\log X)^{2}
		\ll HX^{\frac{1}{2}}C^{-1}. 
	\]
	
	For $M$, from Lemma \ref{Lem_7}, we have 
	\begin{equation} \label{M_result}
		M 
		= \frac{\zeta(\frac{3}{2})}{\zeta(3)} HX^{\frac{1}{2}} + O(HX^{\frac{1}{2}}B^{-1}). 
	\end{equation}
	
	We first consider $R_{1}$. From the fundamental theorem of calculus, 
	\[
		(X+H)^{\rho +\frac{1}{2}} - X^{\rho + \frac{1}{2}}
		 = \bigg( \rho +\frac{1}{2} \bigg) \int_{X}^{X+H} u^{\rho- \frac{1}{2}} du
		 \ll |\gamma| HX^{\beta -\frac{1}{2}}. 
	\]
	By the Stirling's formula and dissecting dyadically, 
	\begin{align}
		R_{1} 
		\ll HX^{-\frac{1}{2}} \sum_{\substack{\rho \\ | \gamma | \le T}} \frac{X^{\beta}}{|\gamma|^{\frac{1}{2}}} 
		&\ll HX^{-\frac{1}{2}} \sum_{1 \le K \le T} K^{-\frac{1}{2}} \sum_{\substack{\rho \\ K < | \gamma | \le 2K}} X^{\beta} \notag \\
		&\ll  HX^{-\frac{1}{2}} \log X \sup_{1 \le K \le T} K^{-\frac{1}{2}} \sum_{\substack{\rho \\ K < | \gamma | \le 2K}} X^{\beta}. \label{R1}
	\end{align}
	For $1 \le K \le T$, we denote 
	\[
		K=X^{\delta}, 
		\quad
		XH^{-1} = X^{\Delta}. 
	\]
	From (\ref{T}), 
	\begin{equation} \label{delta_condition}
		0 \le \delta \le \varepsilon_{1} + \Delta.
	\end{equation}
	By using Lemma \ref{Lem_13}, we have
	\[
		 K^{-\frac{1}{2}} \sum_{\substack{\rho \\ K < | \gamma | \le 2K}} X^{\beta}
		 \ll ( X^{\phi(\delta) -\frac{1}{2}\delta} + X^{1 -\eta +(2\eta -\frac{1}{2}) \delta}) (\log X)^{A}. 
	\]
	From Lemma \ref{Lem_14}, for sufficiently large $X$, we have 
	\[
		\frac{d}{d\delta} (\phi(\delta) -\frac{1}{2}\delta ) >0, 
		\quad
		2\eta -\frac{1}{2} <0. 
	\]
	Then, by (\ref{delta_condition}) and the mean value theorem, we get 
	\begin{align*}
		K^{-\frac{1}{2}} \sum_{\substack{\rho \\ K < | \gamma | \le 2K}} X^{\beta}
		 &\ll ( X^{\phi(\Delta +\varepsilon_{1}) -\frac{1}{2}\Delta +\varepsilon_{1}} + X^{1 -\eta}) (\log X)^{A} \\
		 &\ll ( X^{\phi(\Delta) -\frac{1}{2}\Delta +\varepsilon_{1}} + X^{1 -\eta}) (\log X)^{A}. 
	\end{align*}
	By substituting the above into (\ref{R1}),  we have 
	\begin{equation} \label{R1_result}
		R_{1} 
		\ll HX^{-\frac{1}{2} + \phi(\Delta) -\frac{1}{2}\Delta +\varepsilon_{1}} (\log X)^{A} + HX^{\frac{1}{2}}C^{-1}. 
	\end{equation}
	
	Next, we consider $R_{2}$. From the fundamental theorem in calculus,
	\[
		R_{2} 
		\ll \sum_{\substack{\rho \\ K < | \gamma | \le 2K}} \int_{X}^{X+H} u^{\beta -1} 
		\ll HX^{-1} \sum_{\substack{\rho \\ K < | \gamma | \le 2K}} X^{\beta} . 
	\]
	For $|\gamma| \le T \le X$, we have $1 \le \frac{X}{|\gamma|}$ and
	\begin{equation} \label{R2_result}
		R_{2}
		\ll HX^{-\frac{1}{2}} \sum_{\substack{\rho \\ | \gamma | \le T}} \frac{X^{\beta}}{|\gamma|^{\frac{1}{2}}} 
		\ll HX^{-\frac{1}{2} + \phi(\Delta) -\frac{1}{2}\Delta +\varepsilon_{1}} (\log X)^{A} + HX^{\frac{1}{2}}C^{-1}. 
	\end{equation}
	
	Finally we consider $R_{3}$. By dissecting dyadically, 
	\[
		R_{3} 
		\ll \log X \sup_{1 \le K \le T} K^{-\frac{1}{2}} \sum_{\substack{\rho \\ K < | \gamma | \le 2K}} (HK)^{\beta}. 
	\]
	Similarly to $R_{1}$, we denote $K=X^{\delta}, XH^{-1} = X^{\Delta}$. 
	From (\ref{T_condition}), 
	\[
		X^{\Delta +\varepsilon_{1}} 
		= X^{1+ \varepsilon_{1}} H^{-1} 
		=T 
		\le X. 
	\]
	Thus, we have
	\begin{equation} \label{Delta_condition}
		0 \le \Delta \le 1 -\varepsilon_{1}. 
	\end{equation}
	We define
	\[
		\lambda 
		= \lambda (\delta) 
		= \frac{\log K}{\log HK} 
		= \frac{\delta}{1 -\Delta + \delta}.
	\]
	From the above, $\lambda$ is increasing in (\ref{Delta_condition}). 
	By Lemma \ref{Lem_13}, we have 
	\[
		K^{-\frac{1}{2}} \sum_{\substack{\rho \\ K < | \gamma | \le 2K}} (HK)^{\beta} 
		\ll \bigg( (HK)^{\phi(\lambda) -\frac{1}{2} \lambda} + (HK)^{1 -\eta + (2\eta -\frac{1}{2}) \lambda} \bigg) (\log X)^{A} . 
	\]
	Since $HK = X(XH^{-1})^{-1}K = X^{1 -\Delta +\delta}$, we have
	\begin{align}
		&K^{-\frac{1}{2}} \sum_{\substack{\rho \\ K < | \gamma | \le 2K}} (HK)^{\beta} \notag\\
		\ll & \bigg( X^{(1 -\Delta +\delta)(\phi(\lambda) -\frac{1}{2} \lambda)} + X^{(1 -\Delta +\delta)(1 -\eta + (2\eta -\frac{1}{2}) \lambda)} \bigg) (\log X)^{A} . \label{R3}
	\end{align}
	 For the frist trem of (\ref{R3}), $(1 -\Delta +\delta)$ and $(\phi(\lambda) -\frac{1}{2})$ are increasing with respect to $\delta$, therefore we have 
	\begin{align*}
		X^{(1 -\Delta +\delta)(\phi(\lambda) -\frac{1}{2} \lambda)}
		&\ll X^{(1 +\varepsilon_{1})(\phi(\lambda (\Delta +\varepsilon_{1})) -\frac{1}{2} \lambda(\Delta +\varepsilon_{1}))} \\
		&\ll X^{\phi(\lambda (\Delta +\varepsilon_{1})) -\frac{1}{2} \lambda(\Delta +\varepsilon_{1}) +\varepsilon_{1}} .
	\end{align*}
	By $\lambda '(\delta) \le 1$, Lemma \ref{Lem_14} and the mean value theorem give  
	\[
		\phi(\lambda (\Delta +\varepsilon_{1})) -\frac{1}{2} \lambda(\Delta +\varepsilon_{1})
		\le \phi(\lambda (\Delta)) -\frac{1}{2} \lambda(\Delta) +\varepsilon_{1}
		\le \phi(\Delta) -\frac{1}{2} \Delta +\varepsilon_{1}.
	\]
	For the second term of (\ref{R3}), we have 
	\[
		(1 -\Delta +\delta)(1 -\eta + (2\eta -\frac{1}{2}) \lambda)
		\ll (1 -\Delta)(1 -\eta) + (\frac{1}{2} +\eta)(\Delta +\varepsilon_{1}). 
	\]
	Thus we get 
	\[
		X^{(1 -\Delta +\delta)(1 -\eta + (2\eta -\frac{1}{2}) \lambda)}
		\ll H^{\frac{1}{2}} X^{\frac{1}{2} -\eta + (\frac{1}{2} +\eta) \varepsilon_{1}}
		\ll H^{\frac{1}{2}} X^{\frac{1}{2} +2\varepsilon_{1}}. 
	\]
	Then, 
	\begin{equation} \label{R3_result}
		R_{3} 
		\ll X^{\phi(\Delta) -\frac{1}{2} + 3\varepsilon_{1}} +  H^{\frac{1}{2}} X^{\frac{1}{2} +2\varepsilon_{1}}. 
	\end{equation}
	
	From (\ref{M_result}), (\ref{R1_result}), (\ref{R2_result}) and (\ref{R3_result}), we have
	\[
		\sum_{X < N \le X+H} \widetilde{R}_{B}(N) 
			= \frac{\zeta(\frac{3}{2})}{\zeta(3)}HX^{\frac{1}{2}} + O ( HX^{\frac{1}{2}}C^{-1} + E)
	\]
	provided
	\[
		X^{\frac{1}{4} \varepsilon} \le H \le X^{1 -\varepsilon} 
	\]
	and $\varepsilon_{1} \le \frac{\varepsilon}{16}$, 
	where
	\[
		E = HX^{-\frac{1}{2} +\phi(\Delta) \frac{1}{2} \Delta +4\varepsilon_{4}} + X^{\phi(\Delta) -\frac{1}{2} \Delta + 4\varepsilon_{1}}, 
		\quad
		XH^{-1} = X^{\Delta}. 
	\]
	If $X^{\frac{32-4\sqrt{15}}{49}+\varepsilon} \le H$, we have
	\[
		0 \le \Delta \le \frac{17+4\sqrt{15}}{49}.
	\]
	Then, from Lemma \ref{Lem_16}, we have 
	\[
		E 
		\ll HX^{\frac{1}{2} -\frac{\varepsilon}{10} +4\varepsilon_{1}}. 
	\]
	By taking $\varepsilon_{1} = \frac{\varepsilon}{80}$, we have the asymptotic formula (\ref{result}) provide
	\[
		X^{\frac{32-4\sqrt{15}}{49}+\varepsilon} \le H \le X^{1- \varepsilon}. 
	\]
	Hence, Theorem \ref{main} follows. 

\subsection*{Acknowledgements}
	The author would like to thank Prof.~Maki Nakasuji for her thoughtful guidance and helpful advice. 
	The author also would like to thank Prof.~Yuta Suzuki for his constructive suggestions. 


\vspace{1em}

\begin{flushleft}
{\textsc{%
\small
Fumi Ogihara\\[.3em]
\footnotesize
Graduate School of Science and Technology, \\
Sophia University, \\
7-1 Kioicho, Chiyoda-ku, Tokyo 102-8554, Japan
}

\small
\textit{Email address}: \texttt{f-ogihara-8o8@eagle.sophia.ac.jp}
}
\end{flushleft}

\vspace{1em}

\end{document}